\documentclass[a4paper,12pt]{article}

\RequirePackage{amsthm,amsmath,amsfonts,amssymb}
\RequirePackage[numbers]{natbib}
\RequirePackage{graphicx}

\usepackage{color,ulem}

\theoremstyle{plain}
\newtheorem{theorem}{Theorem}[section]
\newtheorem{proposition}[theorem]{Proposition}
\newtheorem{lemma}[theorem]{Lemma}
\newtheorem{corollary}[theorem]{Corollary}

\theoremstyle{remark}
\newtheorem{example}[theorem]{Example}

\newtheorem{remark}[theorem]{Remark}
\newtheorem{algorithm}[theorem]{Algorithm}

\begin{document}

\begin{center}

  \Large
  {\bf
    A measure-on-graph-valued diffusion:\\
    a particle system with collisions and their
    applications} 
     
  \normalsize

  \bigskip By \bigskip

  \textsc{Shuhei Mano}

  \smallskip
  
  The Institute of Statistical Mathematics, Japan

  %\bigskip Version: \today

\end{center}

\small

{\bf Abstract.}
A diffusion taking value in probability measures on
a graph with a vertex set $V$, $\sum_{i\in V}x_i\delta_i$,
is studied. The masses on each vertices
satisfy the stochastic differential equation of the form
$dx_i=\sum_{j\in N(i)}\sqrt{x_ix_j}dB_{ij}$
on the simplex, where $\{B_{ij}\}$ are independent
standard Brownian motions with skew symmetry and $N(i)$
is the neighbour of the vertex $i$. A dual Markov chain
on integer partitions to the Markov semigroup associated
with the diffusion is used to show that the support of
an extremal stationary state of the adjoint semigroup is
an independent set of the graph. We also investigate
the diffusion with a linear drift, which gives a killing
of the dual Markov chain on a finite integer lattice.
The Markov chain is used to study the unique stationary
state of the diffusion, which generalizes the Dirichlet
distribution. Two applications of the diffusions are
discussed: analysis of an algorithm to find an independent
set of a graph, and a Bayesian graph selection based on
computation of probability of a sample by using coupling
from the past.

\smallskip


\smallskip

%05C81 RWgraph 60K35 IPS 60J70 diff appl 60J90 coa 65C05 MC

\normalsize

\section{introduction}\label{sect:intr} 

Consider a finite graph $\mathcal{G}=(V,E)$ consisting of
vertices $V=\{1,\ldots,r\}$, $r\in\{2,3,\ldots\}$
and edges $E$. Throughout this paper, a graph is undirected
and connected. The neighbour of the vertex $i\in V$ is
denoted by $N(i):=\{j\in V:j\sim i\}$, where $j\sim i$
means that $i$ and $j$ are adjacent. The degree of
the vertex $i$ is denoted by $d_i:=|N(i)|$, which is
the cardinality of the set $N(i)$. An independent set
of $\mathcal{G}$ is a subset of $V$ such that no two of
which are adjacent. In other words, a set of vertices is
independent if and only if it is a clique in the graph
complement of $\mathcal{G}$.

If two vertices of a graph $\mathcal{G}$ has precisely
the same neighbour, throughout this paper, we call
the graph obtained by identifying these two vertices
with keeping the adjacency a {\it reduced graph} of
$\mathcal{G}$. 

Let $\mathcal{P}(\Delta_{r-1})$ be the totality of
probability measures on the simplex
\[
  \Delta_{r-1}=\{(x_1,\ldots,x_r)
  \in\mathbf{R}^V_{\ge 0}:\sum_{i\in V}x_i=1\}
\]
equipped with the topology of weak convergence.
Itoh et al.~\cite{IMS98} discussed
a diffusion taking value in probability measures on
a graph $\mathcal{G}$, whose state is identified with
a probability measure $\sum_{i=1}^r x_i(t)\delta_i$
and the masses on each vertices
\[
  \{x(t)=(x_1(t),\ldots,x_r(t))\in\Delta_{r-1},P_x:t\ge 0\}
  \in \mathcal{P}(\Delta_{r-1})
\]
starts from $x(0)=x$ and satisfies
the stochastic differential equation of the form
\begin{equation}\label{sde}
  dx_i=\sum_{j\in N(i)}\sqrt{x_ix_j}dB_{ij},
  \quad i\in V,
\end{equation}
where $\{B_{ij}\}_{i\sim j\in V}$ are independent standard
Brownian motions with skew-symmetry $B_{ij}=-B_{ji}$.
The generator of the diffusion $L$ operates on a function
$f\in C_0^2(\Delta_{r-1})$ as
\begin{equation}\label{gen}
  Lf=\frac{1}{2}\sum_{i,j\in V}\sigma_{ij}(x)
  \frac{\partial^2f}{\partial x_i\partial x_j},
\end{equation}
where
$\sigma_{ii}(x)=x_i\sum_{j\in N(i)}x_j$,
$\sigma_{ij}(x)=-x_ix_j$, $i\sim j$, and
$\sigma_{ij}(x)=0$ otherwise, and $C_0^2(\Delta_{r-1})$
is the totality of functions with continuous derivative
up to the second order and compact support in $\Delta_{r-1}$.

We say a face of the simplex $\Delta_{r-1}$
{\it corresponds to} a set of vertices $U\subset V$
if the face is the interior of the convex hall of
$\{e_i:i\in U\}$ denoted by
${\rm int}({\rm conv}\{e_i:i\in U\})$, where
$\{e_1,\ldots,e_r\}$ is the set of standard basis of
the vector space $\mathbf{R}^r$. If $U$ consists of
a single vertex, say $e_i$, ${\rm int}({\rm conv}\{e_i\})$
should be read as $e_i$. An observation on the diffusion
is as follows.

\begin{proposition}\label{prop:start}
  Every point in a face of the simplex $\Delta_{r-1}$
  corresponding to an independent set $V_I\subsetneq V$
  of a graph $\mathcal{G}$ is a fixed point of
  the stochastic differential equation \eqref{sde}.
  Namely,
  \begin{equation}\label{prop:start:eq}
    x \in {\rm int}({\rm conv}\{e_i:i\in V_I\})
  \end{equation}
  is a fixed point.
\end{proposition}

\begin{proof}
  For each $i\in V$, $x_i(t)$ is a martingale with
  respect to the natural filtration generated by
  the diffusion $(x(t),P_x)$. The condition
  \eqref{prop:start:eq} implies $x_i=0$ for each
  $i\notin V_I$, $\mathbf{E}(x_i(t))=0$ and thus
  $x_i(t)=0$, $\forall t\ge 0$ almost surely.
  Then, for each $i\in V_I$, \eqref{sde} reduces to
  $dx_i(t)=0$ and we have $x_i(t)=x_i$, $\forall t\ge 0$.
  Therefore, \eqref{prop:start:eq} is a fixed point
  of \eqref{sde}.
\end{proof}

Itoh et al.~\cite{IMS98} discussed the diffusion
$(x(t),P_x)$ as an approximation of the following
discrete stochastic model. Consider a system of $N$
particles, where each particle is assigned to one
of the vertex in $V$, and a continuous-time Markov chain
\[
  \{(n_1(s),\ldots,n_r(s))\in\mathbf{N}_0^V:
  n_1(s)+\cdots+n_r(s)=N,s\ge 0\},
  \quad \mathbf{N}_0:=\{0,1,\ldots\}.
\]
Here, $n_i(s)$ is the number of particles assigned to
the vertex $i\in V$ at time $s$. At each instant of
time, two of the $N$ particles are chosen at uniformly
random. If particles at $i\sim j$ vertices are chosen,
one of the particles is chosen with equal probabilities
and assigned to another vertex. It causes a transition
from $(i,j)$ to $(i,i)$ or $(j,j)$ with equal
probabilities. This stochastic model seems to have
various applications. Tainaka et al. \cite{TIYA06}
discussed this stochastic model as a model of
a speciation process caused by geography. Ben-Haim
et al. \cite{BKR03} considered a slightly modified
version, in which a transition from $(i-1,i+1)$ to $(i,i)$,
$i\in\{2,\ldots,r-1\}$ occurs on the one-dimensional
graph, where $i\sim i+1$, $i\in\{1,\ldots,r-1\}$. They
called it ``compromise process'', because if we regard
the vertices as political positions, then a transition
is a compromise. The process $x^N(s):=n(s)/N$ converges
weakly to the diffusion $(x(t),P_x)$: if
$x^N(0)\Rightarrow x(0)$, then
$x^N([\cdot N(N-1)/2])\Rightarrow x(\cdot)$ as
$N\to\infty$ in the space of right continuous functions
with left limits (see Theorem 10.3.5 of \cite{EK86}). 

Apart from an approximation of the discrete stochastic
model discussed above, the diffusion $(x(t),P_x)$ seems
to appear in various contexts. The diffusion on
a complete graph appears as an approximation of a quite
different discrete stochastic model, called the Wright--Fisher
model in population genetics, which evolves by
repetition of multinomial sampling of fixed number of
particles (see, e.g., Chapter 10 of \cite{EK86}, for
the details). The class of measure-valued diffusions
are called Fleming--Viot processes. Such diffusions
appear as prior distributions in Bayesian statistics.

As we will see below, the support of an extremal
stationary state of the semigroup associated with
the diffusion $(x(t),P_x)$ is a face of $\Delta_{r-1}$
corresponds to an independent set of $\mathcal{G}$. In
this sense, the diffusion can be regarded as
an independent set finding in graph theory. Some
problems related with independent set finding, such
as the maximum independent set problem, is known to
be NP-hard, so it is believed that there is no
efficient algorithm to solve it. Therefore,
algorithms to find maximal independent sets are
useful to obtain practical solutions. For example,
Luby \cite{Luby86} discussed a parallel algorithm for
finding maximal independent sets, which was derived
from a stochastic algorithm to find a maximal
independent set. His algorithm is based on a step
to find an independent set based on random permutation
executed by $O(|E|)$ processors in time $O(1)$
for large $|E|$.

Let us assume there exists a strongly continuous
Markov semigroup $\{T_t:t\ge 0\}$ associated with
the diffusion $(x(t),P_x)$ governed by the generator
\eqref{gen} such that
\[
  T_t 1=1 \quad \text{and} \quad T_t f \ge 0,~\forall f
  \in C(\Delta_{r-1}) ~\text{satisfying}~ f\ge 0,
\]
where $C(\Delta_{r-1})$ is the totality of continuous
functions on $\Delta_{r-1}$, and
\begin{equation}\label{semi}
  T_t f-f=\int_0^t T_s Lfds, \quad \forall
  f\in C_0^2(\Delta_{r-1}).
\end{equation}
The existence of such semigroup for complete graphs was
proven by Ethier \cite{Ethier76}. For the solution of
the stochastic differential equation \eqref{sde}, we have
\[
  T_t f(x)=\mathbf{E}_x\{f(x(t))\}.
\]
Let us denote by $\{T_t^*:t\ge 0\}$ the adjoint
semigroup on $\mathcal{P}(\Delta_{r-1})$ induced by
$\{T_t:t\ge 0\}$. Consider the totality of all fixed
points of $\{T_t^*\}$:
\[
  \mathcal{S}=\{\mu\in\mathcal{P}(\Delta_{r-1}):
  T_t^*\mu=\mu, \forall t\ge 0\}.
\]
We call each element of $\mathcal{S}$
a {\it stationary state} of $\{T_t^*\}$.
A stationary state $\nu$ satisfies
\begin{equation}\label{stat}
  \langle \nu,Lf\rangle=\int Lf(x)\nu(dx)=0,
  \quad \forall f\in C_0^2(\Delta_{r-1}).
\end{equation}
The set $\mathcal{S}$ is non-empty and convex.
The totality of the extremal elements of stationary
states will be denoted by $\mathcal{S}_{\rm ext}$.
Namely, a stationary state $\nu$ is uniquely
represented as
\[
  \nu=\sum_{i=1}^s p_i\nu_i, \quad \nu_1,\ldots,\nu_s
  \in \mathcal{S}_{\rm ext}
\]
for some $p\in\Delta_{s-1}$, $s\in\{2,3,\ldots\}$.
In Theorem~\ref{theo:stat} we see that a support of
an extremal stationary state of the diffusion $(x(t),P_x)$
is a face of $\Delta_{r-1}$ corresponding to
an independent set of $\mathcal{G}$.

In this paper, we will use the term {\it support} in
a sloppy sense. Namely, positivity of a stationary state
is not assumed otherwise stated. In fact,
Proposition~\ref{prop:start} implies that if
the diffusion $(x(t),P_x)$ starts from any point
$x$ in ${\rm int}({\rm conv}\{e_i:i\in V_I\})$,
the stationary state is $\delta_x$, or an atom at $x$.
In this situation, we say the support is
${\rm int}({\rm conv}\{e_i:i\in V_I\})$. In other words,
if a stationary state does not have probability mass
anywhere of an open set, then we say the set is not
the support of the stationary state. Some examples are
as follows.

\begin{example}\label{exam:K}
  Let $\mathcal{G}=K_r$, which is a complete graph consisting
  of $r$ vertices. Since each vertices are maximally
  independent, a stationary state is represented as
  $\sum_{i=1}^rp_i\delta_{{e}_i}$, where
  $\mathcal{S}_{\rm ext}=\{\delta_{{e}_i}:i\in \{1,\ldots,r\}\}$.
  For the solution of the stochastic differential equation
  \eqref{sde}, $p_i$ is the absorption probabilities to
  the vertex $i\in V$. Since $x_i(t)$ is a martingale,
  we know $p_i=x_i(0)$.
\end{example}

\begin{example}\label{exam:C}
  Let $\mathcal{G}=C_r$ for an even positive integer $r$,
  which is a cycle graph consisting of $r$ vertices, i.e.,
  $i\sim i+1$ $\text{mod}~r$. The maximal independent sets
  are the set of all even integer vertices and that of all odd
  integer vertices. When $r=4$, we have independent sets
  $\{1\}$, $\{2\}$, $\{3\}$, $\{4\}$, $\{1,3\}$, and $\{2,4\}$.
  The supports of extremal stationary states are the faces
  $e_i$, $i\in\{1,\ldots,4\}$,
  ${\rm int}({\rm conv}\{e_1,e_3\})$, and
  ${\rm int}({\rm conv}\{e_2,e_4\})$.
  The totality of the extremal stationary states is
  $\mathcal{S}_{\rm ext}=\{\delta_{e_1},\delta_{e_2},\delta_{e_3},\delta_{e_4},\nu_{1,3},\nu_{2,4}\}$, where $\nu_{1,3}$ and $\nu_{2,4}$
  are densities (not necessarily strictly positive) on
  ${\rm int}({\rm conv}\{e_1,e_3\})$ and
  ${\rm int}({\rm conv}\{e_2,e_4\})$, respectively.
  Therefore, a stationary state is represented as
  $\sum_{i=1}^4 p_i\delta_{{e}_i}+p_{1,3}\nu_{1,3}+p_{2,4}\nu_{2,4}$.
\end{example}

\begin{example}\label{exam:Krs}
  Let $\mathcal{G}=K_{r,s}$ for a positive integers $r$ and
  $s$, which is a complete bipartite graph consisting of two
  disjoint and maximal independent sets of $r$ and $s$ vertices.
  For a graph $K_{3,2}$ whose maximal independent sets are
  $\{1,2,3\}$ and $\{4,5\}$, a stationary state may be represented as
  $\sum_{i=1}^5 p_i\delta_{{e}_i}+p_{1,2}\nu_{1,2}+p_{1,3}\nu_{1,3}+p_{2,3}\nu_{2,3}+p_{4,5}\nu_{4,5}+p_{1,2,3}\nu_{1,2,3}$.
\end{example}

To obtain an explicit expression for the stationary states
is a challenging problem. Itoh et al. \cite{IMS98}
successfully obtained an explicit expression
for the stationary states for a star graph $S_2$, where
a star graph $S_{r-1}$, $r\ge 3$ is a complete bipartite
graph $K_{1,r-1}$, and the vertices of $S_{r-1}$ will be
numbered such that $1\sim i$, $i\in\{2,\ldots,r\}$.
A stationary state may be represented as
$\sum_{i=1}^3p_i\delta_{e_i}+p_{2,3}\nu_{2,3}$.
If we identify the vertices $\{2,3\}$, the star graph $S_2$
is reduced to a complete graph $K_2$. With using arguments for
a complete graph in Example~\ref{exam:K}, we know
$p_1=x_1$ and $p_2+p_{2,3}\nu_{2,3}+p_3=x_2+x_3$.
Itoh et al~\cite{IMS98} obtained an explicit expression
for the diffusion starts from
$x\notin \{e_1,e_2,e_3,{\rm int}({\rm conv}\{e_2,e_3\})\}$:
\begin{equation}\label{p_S2}
  p_i=\frac{x_1}{2}
  \left\{\frac{2-x_1}{\sqrt{(2-x_1)^2-4x_i}}-1\right\},
  \qquad i\in\{2,3\}
\end{equation}
by using martingales introduced in Section~\ref{sect:mom}.
This result is for a specific graph, but can be applied to
other graphs reducible to $S_2$. For example, the four-cycle graph
$C_4$ discussed in Example~\ref{exam:C} can be reduced to
$S_2$. Explicit expressions for $p_i$, $i\in\{1,2,3,4\}$
are immediately obtained.

This paper is organized as follows.
In Section~\ref{sect:mom} the martingales used
by Itoh et al. \cite{IMS98} are revisited in
a slightly generalized form. An interpretation of
the martingales is presented in Section~\ref{sect:dual}.
A duality relation between the Markov semigroup
associated with the diffusion and a Markov chain
on the set of ordered integer partitions is
established. The dual Markov chain is studied
and used to show that the support of an extremal
stationary state of the adjoint semigroup is
an independent set of the graph.
Section~\ref{sect:drift} we investigate
the diffusion with a linear drift, which gives
a killing of the dual Markov chain on a finite
integer lattice. The Markov chain is studied
and used to study the unique stationary state
of the diffusion, which generalizes the Dirichlet
distribution.
In Section~\ref{sect:appl} two applications of
the diffusions are discussed: analysis of
an algorithm to find an independent set of
a graph, and a Bayesian graph selection based on
computation of probability of a sample by using
coupling from the past. Section~\ref{sect:disc}
is devoted to discussion of open problems.

\section{Invariants among moments}\label{sect:mom}

For a graph $\mathcal{G}=(V,E)$, $r=|V|$, an element
$a\in\mathbf{N}_0^V$ with $|a|:=a_1+\cdots+a_r<\infty$ will
be denoted by
$a=a_1e_1+\cdots+a_re_r$. We will use multi-index notation;
a monomial $\prod_{i}x_i^{a_i}$ is simply written as $x^a$.

For star graphs $S_{r-1}$ Itoh et al.~\cite{IMS98} noticed
the following homogeneous polynomials of arbitrary order
$n\ge r$:
\[
  \sum_{a_2+\cdots+a_r=n, a\ge 1}
  \left(\begin{array}{c}n-r+1\\a-1\end{array}\right)
  \left(\begin{array}{c}n\\a\end{array}\right)(cx(t))^a,
  \quad
  \left(\begin{array}{c}n\\a\end{array}\right)
  =\frac{n!}{a_2!\cdots a_r!}
\]
are martingales, where the sum is taken over all ordered
positive integer partitions of $n$ satisfying
$a_2+\cdots+a_r=n$ with $a_i\ge 1$, $i\in\{2,\ldots,r\}$
and $c_2+\cdots+c_r=0$ with $c_i\in\mathbf{R}$,
$i\in\{2,\ldots,r\}$. This result is generalized for
a generic graph; an example is a reducible graph in which
vertices in an independent set can be identified
(a reduced graph is defined in Section~\ref{sect:intr}).

\begin{proposition}\label{prop:mar}
  Let $V_I\subsetneq V$ be an independent set of a graph
  $\mathcal{G}$ sharing an adjacent vertex. The homogeneous
  polynomials of any order $n\ge|V_I|+1$:
  \[
    \sum_{\sum_{i\in V_I} a_i=n,a\ge 1}
    \left(\begin{array}{c}n-|V_I|\\a-1\end{array}\right)
    \left(\begin{array}{c}n\\a\end{array}\right)(cx(t))^a,
    \quad
    \left(\begin{array}{c}n\\a\end{array}\right)
    =\frac{n!}{\prod_{i\in V_I}a_i!}
  \]
  is a martingale with respect to the natural filtration
  generated by the solution $(x(t),P_x)$ of the stochastic
  differential equation \eqref{sde}, where
  $\sum_{i\in V_I}c_i=0$ with $c_i\in\mathbf{R}$, $i\in V_I$.
\end{proposition}

\begin{proof}
  Applying It\^o's formula to monomials
  $x^a=\prod_{i\in V_I}x_i^{a_i}$, we have
  \[
    dx^a=\frac{1}{2}\sum_{i\in V_I}a_i(a_i-1)x^{a-e_i+e_j},
  \]
  where the vertex $j$ is adjacent to all vertices of
  $V_I$. Then,
  \begin{equation}
    d\sum_{\sum_{k\in V_I} a_k=n,a\ge 1}f(n,a)x^a
    =\frac{x_j}{2}\sum_{\sum_{k\in V_I} a_k=n,a\ge 1}
    \sum_{i\in V_I}a_i(a_i-1)x^{a-e_i}f(n,a),
    \label{prop:mar:eq1}
  \end{equation}
  where
  \[
    f(n,a)=
    \left(\begin{array}{c}n-|V_I|\\a-1\end{array}\right)
    \left(\begin{array}{c}n\\a\end{array}\right)c^a.
  \]
  The right hand side of the equation \eqref{prop:mar:eq1}
  is proportional to
  \[
    \sum_{i\in V_I}c_i\sum_{\sum_{k\in V_I}a_k=n,a\ge 1,a_i\ge 2}
    x^{a-e_i}f(n-1,a-e_i),
  \]
  and it vanishes because the second summation does not  
  depend on the index $i$ and $\sum_{i\in V_I}c_i=0$.
\end{proof}

Proposition~\ref{prop:mar} gives invariants among $n$-th
order moments of the marginal distribution of
the solution of the stochastic differential equation
\eqref{sde} at a given time. More precisely, such a moment
is represented as
\begin{equation}\label{moment}
  m_a(t):=\mathbf{E}_x\{(x(t))^a\}=T_tx^a, \quad
  |a|=a_1+\cdots+a_r=n, \quad a\in\mathbf{N}_0^V.
\end{equation}
Itoh et al.~\cite{IMS98} used the invariants to derive
the expression \eqref{p_S2} for masses on atoms in
the star graph $S_2$.

\begin{corollary}\label{coro:mar}
  Let $V_I\subsetneq V$ be an independent set of a graph
  $\mathcal{G}$ sharing an adjacent vertex. For moments of
  each order $n\ge|V_I|+1$, we have
  \begin{align*}
    &\sum_{\sum_{i\in V_I} a_i=n,a\ge 1}
    \left(\begin{array}{c}n-|V_I|\\a-1\end{array}\right)
    \left(\begin{array}{c}n\\a\end{array}\right)c^am_a(t)
    \nonumber\\
    &=\sum_{\sum_{i\in V_I} a_i=n,a\ge 1}
    \left(\begin{array}{c}n-|V_I|\\a-1\end{array}\right)
    \left(\begin{array}{c}n\\a\end{array}\right)(cx)^a,
    \quad \forall t>0,
  \end{align*}
  where $\sum_{i\in V_I}c_i=0$ with $c_i\in\mathbf{R}$,
  $i\in V_I$.
\end{corollary}

A small example is as follows.

\begin{example}
  Let $\mathcal{G}=C_4$, which is the cycle graph consisting
  of four vertices discussed in Example~\ref{exam:C}.
  A maximal independent set of $C_4$, $V_I=\{2,4\}$, shares
  $1$ or $3$ as the adjacent vertices. The totality of
  ordered positive integer partitions of four are
  $(a_2,a_4)=(1,3)$, $(2,2)$, and $(3,1)$, which correspond to
  the fourth order moments $m_{e_2+3e_4}(t)$, $m_{2e_2+2e_4}(t)$,
  and $m_{3e_2+e_4}(t)$, respectively. They constitute
  an invariant:
  \[
    m_{e_2+3e_4}(t)-3m_{2e_2+2e_4}(t)+m_{3e_2+e_4}(t)=
    x_2x_4(x_4^2-3x_2x_4+x_2^2), \quad \forall t\ge 0.
  \]
\end{example}

The existence of invariants among same order moments is
interesting, but we are also interested in computation
of each moments. They can be computed by simple algebra,
since each order moments satisfy a system of differential
equations:
\begin{equation}\label{mom}
  \frac{d}{dt}m_a(t)=
  \sum_{i\in V}\sum_{j\in N(i)}
  \frac{a_i(a_i-1)}{2}m_{a-e_i+e_j}-
  \sum_{i\in V}\sum_{j\in N(i)}\frac{a_ia_j}{2}m_a,
  m_a(0)=x^a
\end{equation}
for each $a\in\Pi_{n,r}^{\ge 0}$, where $\Pi_{n,r}^{\ge 0}$
is the totality of the ordered positive integer partitions
of an integer $n$ with $r$ positive integers:
\[
  \Pi_{n,r}^{\ge 0}:=\{a\in\mathbf{N}_0^r:a_1+\cdots+a_r=n\}.
\]
However, it is obvious that solving the system \eqref{mom}
becomes prohibitive as the cardinality of the set
$\Pi_{n,r}^{\ge 0}$ grows. Computation of moments via
stochastic simulation will be discussed in
Section~\ref{sect:appl}.

\section{Dual process on integer partitions}\label{sect:dual}

To study diffusions $(x(t),P_x)$ governed by the generator
\eqref{gen}, we employ a tool called duality, which is a familiar
tool in study of interacting particle systems (see, e.g.,
Chapter 2 of \cite{Liggett85}).

Consider a graph $\mathcal{G}=(V,E)$ and let
$(a(t),\mathbf{P}_a)$, $a(0)=a$ be a continuous-time
Markov chain on the set of ordered non-negative integer
partitions of $n$ with $r=|V|$, which is denoted by
$\Pi_{n,r}^{\ge 0}$, by the rate matrix $\{R_{a,b}\}$:
\begin{align}
  R_{a,a-e_i+e_j}&=\frac{a_i(a_i-1)}{2},
  &i \in V,~j\in N(i),\nonumber\\
  R_{a,b}&=0, \quad & \text{for~all~other}~b\neq a,
  \nonumber\\
  R_{a,a}&=-\sum_{b\neq a} R_{a,b},
  \label{rate}
\end{align}
where
\[
R_{a,b}=\lim_{t\downarrow 0}
\frac{\mathbf{P}_a(a(t)=b)-\delta_{a,b}}{t}, \quad
\forall a,b\in\Pi_{n,r}^{\ge 0}.
\]
The backward equation for the transition probability
$\mathbf{P}_a(a(t)=\cdot)$ is
\begin{equation}
  \frac{d}{dt}\mathbf{P}_a(a(t))
  =\sum_{i\in V}\sum_{j\in N(i)}
  \frac{a_i(a_i-1)}{2}\mathbf{P}_{a-e_i+e_j}(a(t))    
  +R_{a,a}\mathbf{P}_{a}(a(t)).
  \label{kbe}  
\end{equation}

We have the following duality relation between the Markov
semigroup $\{T_t\}$ and the Markov chain
$(a(t),\mathbf{P}_a)$.

\begin{lemma}\label{lemm:dual}
  The Markov semigroup $\{T_t\}$ associated with
  the generator \eqref{gen} and the Markov chain
  $(a(t),\mathbf{P}_a)$ with the rate matrix \eqref{rate}
  satisfy
  \begin{equation}
    T_t x^a=\mathbf{E}_x\{(x(t))^a\}=
    \mathbf{E}_a\left[x^{a(t)}
    \exp\left\{-\int_0^tk(a(s))ds\right\}\right],
    \quad \forall t\ge 0
    \label{dual}
  \end{equation}
  for each $a\in\Pi_{n,r}^{\ge 0}$, where the killing rate is
  \begin{equation}
    k(a)=\sum_{i\in V}\sum_{j\in N(i)}\frac{a_ia_j}{2}
    -\sum_{i\in V}\frac{d_ia_i(a_i-1)}{2}.
    \label{killing}
  \end{equation}
\end{lemma}

\begin{proof}
  Noting that
  \begin{align*}
    Lx^a
    &=\sum_{i\in V}\sum_{j\in N(i)}
      \frac{a_i(a_i-1)}{2}x^{a-e_i+e_j}-
      \sum_{i\in V}\sum_{j\in N(i)}\frac{a_ia_j}{2}x^a\\
    &=\sum_{i\in V}\sum_{j\in N(i)}
      \frac{a_i(a_i-1)}{2}(x^{a-e_i+e_j}-x^a)
      -\sum_{i\in V}\sum_{j\in N(i)}\frac{a_ia_j}{2}x^a
      +\sum_{i\in V}\frac{d_ia_i(a_i-1)}{2}x^a\\
    &=\sum_{b\in\Pi_{n,r}^{\ge 0}}R_{a,b}x^b-k(a)x^a
  \end{align*}
  and the operation \eqref{semi}, we see that
  $g(t,a)=T_tx^a$ satisfies the differential equation
  \[\frac{d}{dt}g(t,a)
     =\sum_{b\in I}R_{a,b}g(t,a)-k(a)g(t,a),\quad
    g(0,a)=x^a
  \]
  for each $a\in\Pi_{n,r}^{\ge 0}$. This is uniquely solved
  by means of the Feynman-Kac formula and the assertion
  follows.
\end{proof}

Since the total number of particles is kept, i.e.
$|a(t)|=a_1(t)+\cdots+a_r(t)=|a|$, $\forall t$,
the killing rate \eqref{killing} is bounded. The killing
rate is not positive definite, however, a key observation
is that if the support of a vector $a$ denoted by
${\rm supp}(a):=\{i\in V:a_i>0\}$ is an independent set
of $\mathcal{G}$, then the killing rate is non-positive:
$k(a)\le 0$. The converse is not always true.

To illustrate the Markov chain $(a(t),\mathbf{P}_a)$,
let us ask specific questions. What is the moment
of $2e_2+e_4$ for the cycle graph $C_4$ discussed
in Example~\ref{exam:C}? For a chain starts from
$2e_2+e_4$, there are two possible transitions;
the one is absorbed into the state $e_1+e_2+e_4$ and
the other is absorbed into the state $e_2+e_3+e_4$, where
the rates are unities (see Figure 1).
\begin{figure}
  \begin{center}
  \includegraphics[width=0.5\textwidth]{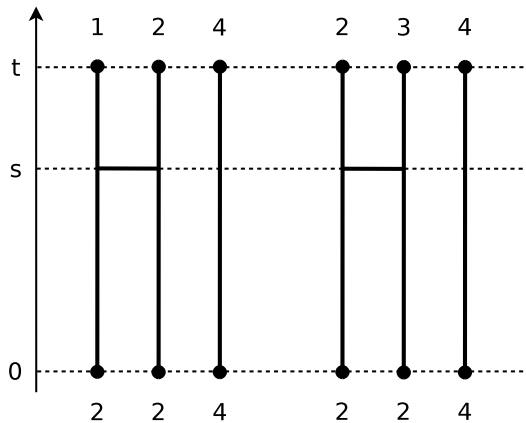}
  \caption{Possible transitions of the chain
    $(a(t),\mathbf{P}_a)$ on the cycle graph $C_4$
    starts from $a=2e_2+e_4$.} 
  \label{fig1}
  \end{center}
\end{figure}
The waiting time for the occurrence of either of
these two transitions follows
the exponential distribution of rate two. Since $k(a)=-2$ and
$k(e_1+e_2+e_4)=k(e_2+e_3+e_4)=2$, the right hand side of
the duality relation \eqref{dual} can be computed as
\begin{align*}
  &x_2^2x_4e^{2t}e^{-2t}+\frac{1}{2}(x_1x_2x_4+x_2x_3x_4)
  \int_0^t e^{2s-2(t-s)}2e^{-2s}ds\\
  &=x_2^2x_4+\frac{(x_1+x_3)x_2x_4}{2}(1-e^{-2t}),
  \quad t\ge 0,
\end{align*}
where $s$ is the time that one of the two possible transitions
occurs. The first term corresponds to the case that no
transition occurs until time $t$. Let us call a transition
event {\it collision}.

\begin{remark}\label{rema:coa}
  Analogous dual Markov chain of the diffusion approximation
  of a kind of Wright--Fisher model was effectively used by
  Shiga \cite{Shiga80a,Shiga80b}, where a transition
  $a\to a-e_i$ occurs with the same rate as in \eqref{rate}.
  Such a transition event is called ``coalescent''.
  In contrast to a collision, the total number of particles
  decreases by a coalescent.
\end{remark}

Here, we have a simple observation about invariants
among moments discussed in Section~\ref{sect:mom}. If
a chain  $(a(t),\mathbf{P}_a)$ starts from a state $a$
such that ${\rm supp}(a)$ is a maximal independent set
$V_I$ of the graph $\mathcal{G}$, then the killing rate
is non-positive and the duality relation \eqref{dual}
is reduced to
\[
  m_a(t)=x^a+
  \mathbf{E}_a\left[x^{a(t)}
  \exp\left\{-\int_0^tk(a(s))ds\right\};
  \text{collisions}~\text{occur}\right].
\]
Corollary~\ref{coro:mar} implies cancellation of the second
term among moments. Moreover, considering a case that
the diffusion $(x(t),P_x)$ starts from a point
$x$ in the face corresponding to $V_I$, by
Proposition~\ref{prop:start}, we know that after a collision
${\rm supp}(a(t))$ must contain a vertex which is not
contained in $V_I$.

Let us ask another question. What is the moment of
$a=2e_1+e_2+e_3$ for the star graph $S_2$ discussed in
Section~\ref{sect:intr}?  Some consideration reveals that
a chain $(a(t),\mathbf{P}_a)$ starts from $a$ never be
absorbed, and transitions occur among the three states:
$a$, $b=e_1+2e_2+e_3$, and $c=e_1+e_2+2e_3$.
Since $k(a)=2$ and $k(b)=k(c)=1$, the duality relation
\eqref{dual} gives
\begin{equation}\label{dual_q2}
  m_a(t)=e^{-t}\mathbf{E}_a
  \left[x^{a(t)}\exp\left\{-\int_0^t1(a(s)=a)ds\right\}\right],
\end{equation}
where $1(\mathcal{A})$ is 1 if the argument $\mathcal{A}$
is true and zero otherwise. Solving the backward equation
\eqref{kbe}, we immediately obtain
\begin{equation}
  \mathbf{P}_a(a(t)=a)=\frac{1}{3}+\frac{2}{3}e^{-3t},\quad
  \mathbf{P}_a(a(t)=b)=\mathbf{P}_a(a(t)=c)
  =\frac{1}{3}-\frac{1}{3}e^{-3t}.
  \label{fourth}
\end{equation}
%\[
%  \left(
%    \begin{array}{c}
%      \mathbf{P}_a(a(t)=a)\\
%      \mathbf{P}_a(a(t)=b)\\
%      \mathbf{P}_a(a(t)=c)
%    \end{array}
%  \right)=
%  \left(
%    \begin{array}{c}
%      \frac{1}{3}\\\frac{1}{3}\\\frac{1}{3}\\ 
%    \end{array}
%  \right)+
%  e^{-t}\left(
%  \begin{array}{c}
%    0\\\frac{a_1}{2}+a_2-\frac{1}{2}\\
%    \frac{a_1}{2}+a_3-\frac{1}{2}
%    \end{array}
%  \right)+
%    e^{-3t}\left(
%  \begin{array}{c}
%    a_1-\frac{1}{3}\\
%    -\frac{a_1}{2}+\frac{1}{6}\\
%    -\frac{a_1}{2}+\frac{1}{6}
%    \end{array}
%  \right).
%\]  
However, computation of the right hand side of the equation
\eqref{dual_q2} seems not easy, because the expectation
depends on a sample path of the chain $(a(s):0\le s\le t)$.
Nevertheless, the moments can be obtained easily by
solving the system of differential equations \eqref{mom}.
In fact, we have
\[
  \left(
    \begin{array}{c}m_{a}(t)\\m_{b}(t)\\m_{c}(t)\end{array}
  \right)
  =\left(
    e^{-(3-\sqrt{3})t}(A+B)+e^{-2t}C+e^{-(3+\sqrt{3})t}(A-B)
  \right)
  \left(
    \begin{array}
      {c}x_1^2x_2x_3\\x_1x^2_2x_3\\x_1x_2x_3^2
    \end{array}
  \right),
\]
where
\[
  A=\left(\begin{array}{rrr}
  \frac{1}{2}&0&0\\  
  0&\frac{1}{4}&\frac{1}{4}\\  
  0&\frac{1}{4}&\frac{1}{4}\\  
  \end{array}\right), \quad
  B=\frac{\sqrt{3}}{6}
  \left(\begin{array}{rrr}
  -1&1&1\\  
  1&\frac{1}{2}&\frac{1}{2}\\  
  1&\frac{1}{2}&\frac{1}{2}\\  
  \end{array}\right), \quad 
  C=\left(\begin{array}{rrr}
  0&0&0\\  
  0&\frac{1}{2}&-\frac{1}{2}\\  
  0&-\frac{1}{2}&\frac{1}{2}\\
  \end{array}\right).
\]

The observations above lead to the following proposition on
the fate of the Markov chain $(a(t),\mathbf{P}_a)$.

\begin{proposition}\label{prop:abs}
  Consider the Markov chain $(a(t),\mathbf{P}_a)$ on the
  set of the ordered non-negative integer positive
  partitions $\Pi_{n,r}^{\ge 0}$.
  \begin{itemize}
  \item[$(i)$] If a chain starts from a state $a$ satisfying
    $1\le |a|\le r$, then it is absorbed into an element of
    $\{0,1\}^V$.
  \item[$(ii)$] If a chain starts from a state $a$ satisfying
    $n=|a|>r$, then the transition probability $\mathbf{P}_a(a(t)=\cdot)$
    converges to the uniform distribution on the set of
    ordered positive integer partitions
    \[
    \Pi_{n,r}:=\{b\in \mathbf{N}^r:b_1+\cdots+b_r=n\}
    \subsetneq \Pi_{n,r}^{\ge 0}.
    \]
  \end{itemize}  
\end{proposition}  

\begin{proof}
  $(i)$ A state $a$ is absorbing if and only if the row
  vector of the rate matrix \eqref{rate} is zero, which
  implies $a\in\{0,1\}^V$. Consider the set of
  vertices $V_0(a(t))=\{i\in V:a_i(t)=0\}$. Then,
  $z(t):=|V_0(a(t))|$, $t\ge 0$ is a death process and is
  absorbed into the state $r-n$ at a Markov time
  $\tau_0<\infty$ with respect to the Markov chain
  $(a(t),\mathbf{P}_a)$,
  where $a_j(\tau_0)=1$ for $j\in V\setminus V_0(a(\tau_0))$.
  Let us show that the process eventually decreases if
  $z(t)>r-n$. If $z(t)>r-n$, at least one vertex, say
  $j\in V\setminus V_0(a(t))$, satisfies $a_j(t)\ge 2$.
  If the vertex $j$ is connected with a vertex in
  $V_0(a(t))$, say $k$, the transition
  $a(t)\to a(t_+)=a(t)-e_j+e_k$ occurs
  with positive probability and it makes $z(t_+)=z(t)-1$.
  Otherwise, the vertex $j$ should be connected with
  at least one vertex in $V\setminus V_0(a(t))$, say $l$.
  The transition $a(t)\to a(t_+)=a(t)-e_j+e_l$ makes
  $a_l(t_+)\ge 2$. If the vertex $l$ is connected with
  a vertex in $V_0(a(t))$, the assertion is shown. Otherwise,
  the vertex $l$ should be connected with at least one
  vertex in $V\setminus \{V_0(a(t))\cup j\}$. The proof is
  completed by repeating this procedure until we reach
  a vertex in $V\setminus V_0(a(t))$ which is connected
  with a vertex in $V_0(a(t))$.\\
  $(ii)$ If a chain starts from a state in the set
  $\Pi_{n,r}^{\ge 0}\setminus \Pi_{n,r}$, the argument
  for $(i)$, i.e., $z(t)$ is a death process,
  shows that in a finite time the chain reaches
  a state, say $a$, in the set $\Pi_{n,r}$ and never
  exits from $\Pi_{n,r}$. For such a case, consider
  restarting the chain from the state $a$.
  For simplicity, we consider the case of $n=r+1$.
  Then, a state $a(t)$ can be identified with
  the unique vertex $i\in V$ satisfying $a_i(t)=2$.
  Since we are considering a connected graph, there exist
  a path $i=v_0, v_1, v_2, \ldots, v_s=j$
  for any $i\neq j\in V$ with some $s\in\mathbf{N}$, where
  $v_{k-1}\sim v_k$, $k=1,\ldots,s$. Since all of
  the transitions occur with rate unity, the sample path has
  a positive probability. Hence, the Markov chain is
  irreducible and ergodic, and there exists the unique
  stationary distribution on the set $\Pi_{n,r}$. Since
  the uniform distribution on $\Pi_{n,r}$ satisfies
  the backward equation \eqref{kbe}, it is the stationary
  distribution. Cases of $n>r+1$ can be shown in
  a similar manner by showing that up to $n-r$ particles
  can be moved from a vertex to another vertex.
\end{proof}  

Proposition~\ref{prop:abs} shows that if the Markov chain
$(a(t),\mathbf{P}_a)$ starts from a state $a$ satisfying
$|a|>|V|$, convergence of the chain can be divided to
two phases:
(1) exits from the set $\Pi_{n,r}^{\ge 0}\setminus\Pi_{n,r}$,
and (2) converges to the uniform distribution on
the set $\Pi_{n,r}$. The largest eigenvalue of the rate
matrix is zero. To consider the mixing, we have to know
the second largest eigenvalue, say $\lambda_2$.
By the spectral decomposition of the transition probability,
the mixing time of the phase 2, or the infimum of $t$ satisfying
\[
  \max_{a\in\Pi_{n,r}}
  \|\mathbf{P}_a(a(t)=\cdot)-|\Pi_{n,r}|^{-1}\|_{\rm TV}<\epsilon
\]
for $\epsilon>0$ is less than $c\lambda_2^{-1}\log \epsilon^{-1}$
for some constant $c>0$, where $\|\mu-\nu\|_{\rm TV}$ is
the total variation distance between probability measures $\mu$
and $\nu$.

\begin{example}
  Consider the case that $n=r+1$ (see the proof of
  Proposition~\ref{prop:abs} ($ii$)). It can be shown that
  the rate matrix $\{R_{a,b}\}$ reduces to the negative of
  the Laplacian matrix of $\mathcal{G}$, whose second
  largest eigenvalue is called algebraic connectivity of
  $\mathcal{G}$. For a connected graph, it is known that
  the largest eigenvalue is zero and the second largest
  eigenvalue is bounded below by 
  $4/(r{\rm diam}(\mathcal{G}))$ \cite{Mohar91},
  where ${\rm diam}(\mathcal{G})$ is the diameter of
  $\mathcal{G}$, or the maximum of the shortest path lengths
  between any pair of vertices in $\mathcal{G}$.
  For the star graph $S_2$ and the forth order monomials
  discussed above, the rate matrix is the negative of
  the Laplacian matrix:
  \[
    R=-\left(\begin{array}{rrr}
    2 & -1& -1\\
    -1&  1&  0\\
    -1&  0&  1
    \end{array}\right),
  \]
  and the eigenvalues are $0$, $-1(<-2/3)$, and $-3$,
  where $r=3$, ${\rm diam}(S_2)=2$. The two eigenvalues
  appear in the transition probabilities \eqref{fourth}.
  For a generic graph with large $r$, the mixing time is
  $O(r{\rm diam}(\mathcal{G}))\log \epsilon^{-1}$.
\end{example}
  
Assessment of the phase 1 seems harder. The death process
$z(t)=|V_0(a(t))|$, $t>0$ with $z(0)\le r-1$ is not
a Markov process. The death rate is bounded below:
\[
  \sum_{i\in V}\sum_{j\in N(i)\cap V_0(a(t))}\frac{a_i(t)(a_i(t)-1)}{2}\ge
  \frac{a_{k}(t)(a_{k}(t)-1)}{2} ~\text{for~some}~
  k\in V\setminus V_0(a(t)).
\]
To obtain a rough estimate of the right hand side, let us
suppose the state $a(t)$ follows the uniform distribution
on the set of ordered positive integer partitions
$\Pi_{n,r-z}$, where
\[
  \mathbf{P}(a_k=i)=
  \left(\begin{array}{c}n-i\\r-z-2\end{array}\right)
  \left(\begin{array}{c}n-1\\r-z-1\end{array}\right)^{-1},
  \quad i\in\{1,\ldots,n-r+z+1\}.
\]
Then, the expectation of the lower bound is
$(n-r+z)_{r-z+2}/(r-z)_3$, where $(n)_i=n(n+1)\cdots(n+i-1)$.
When $n$ is large, the dominant contribution to the expectation
of the waiting time for the exit comes from the period
$\{t:z(t)=z(0)\}$, and it is $O(n^{z(0)-r-2})$.
Hence, the expectation of the waiting time for the exit
would be $O(n^{-3})$.

Shiga \cite{Shiga80a,Shiga80b} and Shiga and Uchiyama
\cite{SU86} studied structures of extremal stationary
states of the diffusion approximation of a kind of
Wright--Fisher model in $[0,1]^S$ for a countable set
$S$ by using its dual Markov chain. Extremal states of
the adjoint Markov semigroup $\{T^*_t\}$ on
$\mathcal{P}(\Delta_{r-1})$ induced by $\{T_t\}$
associated with the diffusion $(x(t),P_x)$ can be
studied by using the dual Markov chain
$(a(t),\mathbf{P}_a)$. Note that positivity of
a stationary state is not assumed, as explained in
Section~\ref{sect:intr}.

\begin{theorem}\label{theo:stat}
  The support of an extremal stationary state of
  the adjoint Markov semigroup $\{T^*_t\}$ is a face of
  the simplex $\Delta_{r-1}$ corresponding to
  an independent set $V_I$ of the graph $\mathcal{G}$,
  namely, ${\rm int}({\rm conv}\{e_i:i\in V_I\})$.
\end{theorem}

\begin{proof}
  Consider a Markov chain $(a(t),\mathbf{P}_a)$ with
  rate matrix \eqref{rate} starts from a state
  $a(0)=a\in\{0,1\}^V$. According to
  Proposition~\ref{prop:abs} $(i)$, such $a$ is an
  absorbing state and the chain stays there.
  Lemma~\ref{lemm:dual} gives
  \begin{equation*}
    T_tx^a=x^a
    \exp\left(
    -t\sum_{i\in V}\sum_{j\in N(i)}\frac{a_ia_j}{2}
    \right), \quad \forall t\ge 0.
  \end{equation*}
  A stationary state
  $\nu\in\mathcal{P}(\Delta_{r-1})$ satisfies
  $\langle\nu,T_t x^a\rangle=\langle T^*_t\nu,x^a\rangle=\langle\nu,x^a\rangle$
  for any $x^a$. If $a\in\{0,1\}^V$, this condition
  reduces to
  \begin{equation*}\label{theo:stat:eq1}  
    \langle\nu,x^a\rangle
    \left\{1-\exp\left(
    -t\sum_{i\in V}\sum_{j\in N(i)}\frac{a_ia_j}{2}
    \right)\right\}=0, \quad \forall t\ge 0.
  \end{equation*}
  Therefore, if ${\rm supp}(a)$ is not an independent
  set, then $\langle\nu,x^a\rangle=0$. Since we
  are considering a connected graph $\mathcal{G}$,
  $V$ is not an independent set of $\mathcal{G}$.
  The condition reduces to
  $\langle\nu,x_1\cdots x_r\rangle=0$. Since
  \[
    \{x\in\Delta_{r-1}: x_1\cdots x_r>0\}=
    {\rm int}(\Delta_{r-1})=
    {\rm int}({\rm conv}\{e_i: i\in V\}),
  \]
  the condition $\langle\nu,x_1\cdots x_r\rangle=0$ excludes
  ${\rm int}({\rm conv}\{e_i: i\in V\})$ from the support
  of $\nu$. Suppose there exists a vertex $j_1$ such that
  $V^{(1)}=V\backslash\{j_1\}$ is still not an independent
  set. Set $a$ such that $a_i=1$ if $i\in V^{(1)}$ and
  $a_i=0$ otherwise for each $i\in\{1,\ldots,r\}$.
  Since
  \[
    \{x\in\Delta_{r-1}: x^a>0\}=
    {\rm int}({\rm conv}\{e_i: i\in V^{(1)}\}),
  \]
  the condition $\langle\nu,x^a\rangle=0$ excludes
  ${\rm int}({\rm conv}\{e_i: i\in V^{(1)}\})$ from
  the support of $\nu$. Repeating of this procedure
  yields an independent set
  $V_I=V\setminus \{j_1,\ldots,j_s\}$ for
  some $s\in\mathbf{N}$, and the face
  ${\rm int}({\rm conv}\{e_i: i\in V_I\})$
  is not excluded from the support of $\nu$.
\end{proof}

The steps in the above proof appear in the following example.

\begin{example}\label{exam:C2}
  Let $\mathcal{G}=C_4$, which is the cycle graph consisting
  of four vertices. The support of an extremal stationary
  state appeared in Example~\ref{exam:C} is confirmed as follows.
  Remove the vertex $j_1=2$ from the vertex set $V$. Since
  $V^{(1)}=V\setminus\{2\}=\{1,3,4\}$ is not
  an independent set, the face
  ${\rm int}({\rm conv}\{e_1,e_3,e_4\})$ is excluded
  from a support of extremal stationary states.
  Then, remove $j_2=4$ from $V^{(1)}$. Since
  $V^{(2)}=V\setminus\{2,4\}=\{1,3\}$ is an
  independent set, the face ${\rm int}({\rm conv}\{e_1,e_3\})$
  is the support of an extremal stationary state.
\end{example}

A direct consequence of Theorem~\ref{theo:stat} on
the moments is as follows.

\begin{corollary}\label{coro:mom}
  For each $n\in\mathbf{N}$, if the limit of an $n$-th
  order moments of the diffusion $(x(t),P_x)$ on
  a graph $\mathcal{G}$ is positive, namely,
  \[
    \lim_{t\to\infty}m_a(t)
    =\lim_{t\to\infty}\mathbf{E}_x\{(x(t))^a\}>0,
    \quad \text{for} ~~ a ~~ \text{satisfying} ~~ |a|=n,
  \]
  then ${\rm supp}(a)$ is an independent set of
  $\mathcal{G}$.
\end{corollary}  
  
\section{Diffusion with a linear drift}\label{sect:drift}

In this section we consider the diffusion
$(\tilde{x}(t),\tilde{P}_x)$ taking value in probability
measures on a graph $\mathcal{G}=(V,E)$, $r=|V|$
satisfying the following stochastic differential
equation with a linear drift:
\begin{equation}\label{sde_k}
  dx_i=\sum_{j \in N(i)}\sqrt{x_ix_j}dB_{ij}
  +\frac{\alpha}{2}(1-rx_i)dt, \quad i\in V
\end{equation}
for $\alpha\in\mathbf{R}_{>0}$. The drift term,
$\alpha(1-rx_i)dt/2$, gives a killing of the dual process
with a linear rate. As shown below, behaviors of
the diffusion and the dual Markov chain are significantly
different from those without drift discussed in previous
sections.

In Itoh et al.'s discrete stochastic model described in
Section~\ref{sect:intr}, this drift corresponds to
adding the following dynamics: at each instant of time,
one of $N$ particles is chosen at uniformly random and
assigned to another vertex chosen at uniformly random with
rate $\alpha(r-1)/(N-1)$. In the Wright--Fisher model,
this drift corresponds to a mutation mechanism
\cite{EK86}.

Let $(\tilde{a}(t),\tilde{\mathbf{P}}_a)$ be
a continuous-time Markov chain on a finite integer lattice,
or the set of non-negative integers
\[
  I:=\{a\in \mathbf{N}^V_0:a_1+\cdots+a_r<\infty\}
\]
by the rate matrix $\{\tilde{R}_{a,b}\}$:
\begin{align} \label{rate_k}
  \tilde{R}_{a,a-e_i+e_j}&=\frac{a_i(a_i-1)}{2},
  &i \in V,~j\in N(i), \nonumber\\
  \tilde{R}_{a,a-e_i}&=\frac{\alpha}{2}a_i,
  &i \in V, \nonumber\\
  \tilde{R}_{a,b}&=0, \quad & \text{for~all~other}~b
  \neq a, \nonumber\\
  \tilde{R}_{a,a}&=-\sum_{b\neq a}\tilde{R}_{a,b},
\end{align}
where
\[
  \tilde{R}_{a,b}=\lim_{t\downarrow 0}
  \frac{\tilde{\mathbf{P}}_a(\tilde{a}(t)=b)
  -\delta_{a,b}}{t}, \quad
  \forall a,b\in I
\]
The backward equation for the transition probability
$\tilde{\mathbf{P}}_a(\tilde{a}(t)=\cdot)$ is
\begin{align*}
  \frac{d}{dt}\tilde{\mathbf{P}}_a(\tilde{a}(t))
  =&\sum_{i\in V}\sum_{j\in N(i)}
  \frac{a_i(a_i-1)}{2}
  \tilde{\mathbf{P}}_{a-e_i+e_j}(\tilde{a}(t))
  +\frac{\alpha}{2}\sum_{i\in V}
  \tilde{\mathbf{P}}_{a-e_i}(\tilde{a}(t))
  \nonumber\\
  &-\tilde{R}_{a,a}\tilde{\mathbf{P}}_{a}(\tilde{a}(t)).
  \label{kbe_k}  
\end{align*}

The following duality relation between the Markov
semigroup associated with the diffusion
$(\tilde{x}(t),\tilde{P}_x)$ denoted by
$\{\tilde{T}_t:t\ge 0\}$, and the Markov chain
$(\tilde{a}(t),\tilde{\mathbf{P}}_a)$ can be shown in
the same manner as Lemma~\ref{lemm:dual}.

\begin{lemma}\label{lemm:dual_k}
  The Markov semigroup $\{\tilde{T}_t\}$ and
  the Markov chain $(\tilde{a}(t),\tilde{\mathbf{P}}_a)$
  with the rate matrix \eqref{rate_k} satisfy
  \begin{equation}
    \tilde{T}_tx^a=\tilde{\mathbf{E}}_x(x^a(t))=
    \tilde{\mathbf{E}}_a\left[x^{\tilde{a}(t)}
    \exp\left\{-\int_0^t\tilde{k}(\tilde{a}(s))ds\right\}
    \right], \quad t\ge 0
    \label{dual_k}
  \end{equation}
  for each $a\in I$, where the killing rate is
  \begin{equation}
    \tilde{k}(a)=\sum_{i\in V}\sum_{j\in N(i)}\frac{a_ia_j}{2}
    +\frac{\alpha}{2}(r-1)|a|
    -\sum_{i\in V}\frac{d_ia_i(a_i-1)}{2}.
    \label{killing_k}
  \end{equation}
\end{lemma}

In contrast to the rate matrix $\{R_{a,b}\}$ in \eqref{rate},
particles are erased and it makes the total number of
particles
$|\tilde{a}(t)|=\tilde{a}_1(t)+\cdots+\tilde{a}_r(t)$
decreases. It is clear from the rate matrix \eqref{rate_k}
that $0$ is the unique absorbing state.

\begin{proposition}\label{prop:stop}
  Let $\tau:=\inf\{t>0:\tilde{a}(t)=0\}$, which is a Markov
  time with respect to the Markov chain
  $(\tilde{a}(t),\tilde{\mathbf{P}}_a)$ with
  the rate matrix \eqref{rate_k}. Then,
  \[
  \tilde{\mathbf{P}}_a(\tau<\infty)=1 \quad ~\text{and}~ \quad
  \tilde{\mathbf{E}}_a\tau=\frac{2}{\alpha}\sum_{i=1}^{|a|}\frac{1}{i}.
  \]
\end{proposition}

\begin{proof}
  Since $\tilde{a}(t)=0$ if and only if $|\tilde{a}(t)|=0$,
  we consider the Markov chain of the cardinality
  $|\tilde{a}(t)|$. According to the rate matrix
  \eqref{rate_k}, it is a linear death process with rate
  $\alpha|\tilde{a}(t)|/2$. Noting that $\tau$ is
  the convolution of exponential random variables of rates
  $\alpha i/2$, $i\in\{1,\ldots,|a|\}$, we have the assertion.
\end{proof}  

To illustrate the Markov chain $(\tilde{a}(t),\tilde{\mathbf{P}}_a)$,
let us ask a specific question. What is the moment of $a=e_1+e_2$
from the cycle graph $C_4$? For a chain starts from $a$, there
are four possible sample paths (Figure 2):
\begin{figure}
  \begin{center}
  \includegraphics[width=0.65\textwidth]{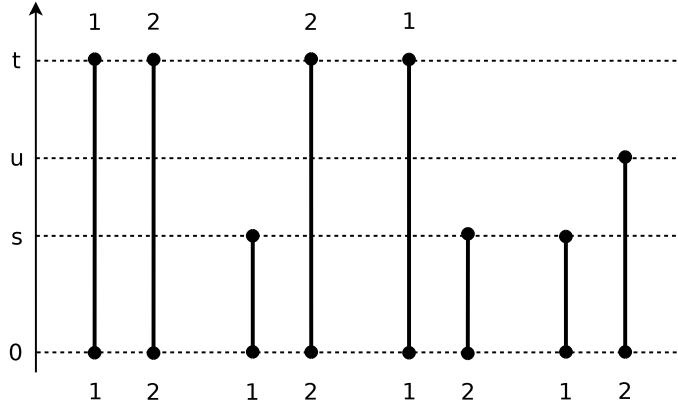}
  \caption{Possible sample paths of the chain
    $(\tilde{a}(t),\tilde{\mathbf{P}}_a)$ on the cycle graph $C_4$
    starts from $e_1+e_2$.}
  \label{fig2}
  \end{center}
\end{figure}
\begin{itemize}
\item[(i)] No particle is erased;
\item[(ii)] The particle 1 is erased but the particle 2
  survives;
\item[(iii)] The particle 2 is erased but the particle 1
  survives;
\item[(iv)] Both particles are erased.
\end{itemize}
The waiting time of either of two particles is erased
follows the exponential distribution of rate $\alpha$.
Since $\tilde{k}(a)=1+3\alpha$ and
$\tilde{k}(e_1)=\tilde{k}(e_2)=3\alpha/2$,
the right hand side of the duality relation \eqref{dual_k}
can be computed as
\begin{align*}
  &x_1x_2 e^{-\alpha t} e^{-(1+3\alpha)t}+
  \frac{x_1+x_2}{2}\int_0^t \alpha e^{-\alpha s}e^{-\alpha (t-s)/2}
  e^{-(1+3\alpha)s-3\alpha(t-s)/2}ds\\
  &+\int_0^t\int_0^u
  \alpha e^{-\alpha s}\frac{\alpha}{2}e^{-\alpha (u-s)/2}
  e^{-(1+3\alpha)s-3\alpha(u-s)/2}dsdu\\
  =&x_1x_2e^{-t(1+4\alpha)}
  +\frac{(x_1+x_2)\alpha}{2(1+2\alpha)}(e^{-2\alpha t}-e^{-(1+4\alpha)t})
  +\frac{\alpha}{4(1+4\alpha)}
  -\frac{\alpha e^{-2\alpha t}}{4(1+2\alpha)}\\
  &+\frac{\alpha^2 e^{-(1+4\alpha)t}}{2(1+2\alpha)(1+4\alpha)},
\end{align*}
where $s>0$ and $u>s$ are the times that a particle
is erased.

The stationary state of the adjoint Markov semigroup
$\{\tilde{T}^*_t\}$ on $\mathcal{P}(\Delta_{r-1})$ induced
by $\{\tilde{T}_t\}$ consists of the unique probability
measure.

\begin{theorem}\label{theo:stat_k}
  For the adjoint Markov semigroup $\{\tilde{T}^*_t\}$,
  there exists the unique stationary state
  $\nu_\alpha\in\mathcal{P}(\Delta_{r-1})$ satisfying
  \[
    \lim_{t\to\infty} \tilde{T}^*_t \delta_x=\nu_\alpha
  \]
  for every $x\in\Delta_{r-1}$.
\end{theorem}

\begin{proof}
  Since the Markov chain $(\tilde{a}(t),\tilde{\mathbf{P}}_a)$
  with the rate matrix \eqref{rate_k} is absorbed into 0,
  Lemma~\ref{lemm:dual_k} and Proposition~\ref{prop:stop}
  give
  \begin{align}
    \lim_{t\to\infty}\tilde{T}_t x^{a}
    =&\lim_{t\to\infty}\tilde{\mathbf{E}}_a
    \left[x^{\tilde{a}(t)}\exp\left\{-\int_0^t \tilde{k}(\tilde{a}(s))ds
    \right\};t\le\tau\right]\nonumber\\
    &+\lim_{t\to\infty}\tilde{\mathbf{E}}_a
    \left[x^{\tilde{a}(t)}\exp\left\{-\int_0^t \tilde{k}(\tilde{a}(s))ds
    \right\};t>\tau\right]\nonumber\\
    =&\tilde{\mathbf{E}}_a\left[
    \exp\left\{-\int_0^\tau \tilde{k}(\tilde{a}(s))ds\right\}
    ;\tau<\infty\right]\nonumber\\
    =&\tilde{\mathbf{E}}_a
    \exp\left\{-\int_0^\tau \tilde{k}(\tilde{a}(s))ds\right\}
    \label{theo:stat_k:eq1}
  \end{align}
  % \tilde{T} opperate on a function. Exclude identically 0. 
  for all $a\in I$. Since
  $\lim_{t\to\infty}\langle\delta_x,\tilde{T}_tx^a\rangle=\lim_{t\to\infty}\langle\tilde{T}^*_t\delta_x,x^a\rangle$ for each $x\in\Delta_{r-1}$,
  there exists a unique probability measure $\nu_\alpha$ satisfying
  $\lim_{t\to\infty}\tilde{T}_tf=\langle\nu_\alpha,f\rangle$
  for all $f\in C(\Delta_{r-1})$.
\end{proof}

The stationary state $\nu_\alpha$ converges weakly to
a common limit as $\alpha\to\infty$ irrespective of
the graph.

\begin{corollary}\label{coro:stat_k}
  The stationary state $\nu_\alpha$ of the adjoint Markov
  semigroup $\{\tilde{T}^*_t\}$ satisfies
  \[
    \lim_{\alpha\to\infty}\nu_{\alpha}=
    \delta_{(1/r,\cdots,1/r)}.
  \]  
\end{corollary}

\begin{proof}
  Since the killing rate \eqref{killing_k} is bounded
  and $\tilde{k}(a)=\alpha(r-1)|a|/2+O(1)$ for large $\alpha$,
  the leading contribution to the expression
  \eqref{theo:stat_k:eq1} can be evaluated by the death
  process $|\tilde{a}(t)|$ considered in
  Proposition~\ref{prop:stop}, whose waiting time follows
  the exponential distribution of rate
  $\alpha|\tilde{a}(t)|/2$. Let $n=|\tilde{a}(0)|=|a|$.
  We have
  \begin{align*}
    \langle\nu_\alpha,x^a\rangle
    &=\tilde{\mathbf{E}}_a
    \exp\left\{-\int_0^\tau \tilde{k}(\tilde{a}(s))ds\right\}\\
    &\le \int_0^\infty ds_1 \cdots ds_n
    \prod_{i=1}^n
    \left\{e^{-\frac{\alpha}{2}(r-1)is_i+c_n}
    \frac{\alpha}{2}ie^{-\frac{\alpha}{2}is_i}\right\}\\
    &=\left(\frac{\alpha}{2}\right)^{n}n!
    \prod_{i=1}^n\frac{2}{\alpha ri+2c_n},  
    \quad\forall a\in I,
  \end{align*}
  where $c_n$ is a constant satisfying
  $\tilde{k}(b)\le \alpha(r-1)|b|/2+c_n$ for all $b$
  satisfying $|b|\in\{1,\ldots,n\}$. In the same way,
  $\langle\nu_\alpha,x^a\rangle$ is bounded below.
  The assertion follows by taking the limit $\alpha\to\infty$
  of these bounds.
\end{proof}

Moreover, the stationary state $\nu_\alpha$ has a continuous
and strictly positive density.

\begin{theorem}\label{theo:stat2}
  For the adjoint Markov semigroup $\{\tilde{T}^*_t\}$,
  the unique stationary state $\nu_\alpha\in\mathcal{P}(\Delta_{r-1})$
  is absolutely continuous with respect to the Lebesgue
  measure on $\Delta_{r-1}$ and admits a probability density
  that is strictly positive in ${\rm int}(\Delta_{r-1})$
  and of $C^\infty(\Delta_{r-1})$-class.
\end{theorem}

\begin{proof}
  We first show that $\nu_\alpha$ has a density of
  $C^\infty(\Delta_{r-1})$-class. By Theorem~\ref{theo:stat_k}, we have
  \begin{align*}
    \langle \nu_\alpha,e^{\sqrt{-1}\sum_{i\in V}x_in_i}\rangle
    &=\sum_{j=0}^\infty\frac{(\sqrt{-1})^j}{j!}
    \sum_{a\in\Pi_{j,r}^{\ge 0}}
    \left(\begin{array}{c}j\\a\end{array}\right)
    \langle\nu_\alpha,(xn)^a\rangle\\
    &=\sum_{j=0}^\infty\frac{(\sqrt{-1})^j}{j!}
    \sum_{a\in\Pi_{j,r}^{\ge 0}}
    \left(\begin{array}{c}j\\a\end{array}\right)n^a
    \tilde{\mathbf{E}}_a
    \exp\left\{-\int_0^\tau \tilde{k}(\tilde{a}(s))ds\right\}
  \end{align*}    
  for each $n\in\mathbf{Z}^V$. Therefore, $\nu_\alpha$ has a
  $C^\infty(\Delta_{r-1})$-density represented as
  \begin{align*}
    p_\alpha(x)
    =&\sum_{n\in\mathbf{Z}^V}
    e^{-\sqrt{-1}\sum_{i\in V}x_in_i}
    \sum_{j=0}^\infty\frac{(\sqrt{-1})^j}{j!}
    \sum_{a\in\Pi_{j,r}^{\ge 0}}
    \left(\begin{array}{c}j\\a\end{array}\right)n^a
    \tilde{\mathbf{E}}_a
    \exp\left\{-\int_0^\tau \tilde{k}(\tilde{a}(s))ds\right\}.
  \end{align*}
  We next show that the density $p_\alpha$ is strictly positive
  in ${\rm int}(\Delta_{r-1})$.
  Consider an approximation of $p_\alpha(x)$ by polynomials:
  \[
    p_\alpha^{(n)}(x)=
    \sum_{i_1=0}^n\cdots\sum_{i_r=0}^n p_\alpha
    \left(\frac{i_1}{n},\ldots,\frac{i_r}{n}\right)
    \left(\begin{array}{c}n\\i\end{array}\right)x^i
  \]
  satisfying $\lim_{n\to\infty}p^{(n)}_\alpha(x)=p_\alpha(x)$.
  Suppose there exist a point
  $\bar{x}\in{\rm int}(\Delta_{r-1})$ satisfying
  $p_\alpha(\bar{x})=0$. Since the monomial $\bar{x}^i$
  is strictly positive, for any small positive constants
  $\epsilon_1$ and $\epsilon_2$ there exist $N$ such that 
  \[
    p_\alpha
    \left(\frac{i_1}{n},\ldots,\frac{i_r}{n}\right)<\epsilon_1,
    \quad \forall i\in\{0,\ldots,n\}^V
  \]
  and ${\rm int}(\Delta_{r-1})$ is covered by
  open balls:
  \[
    B_{n}(i)=\{x:|x_{i_j}-i_j/n|<\epsilon_2\},
    \quad i\in\{0,\ldots,n\}^V
  \]
  for all $n>N$. Since $p_\alpha$ is smooth, for every
  point $x\in{\rm int}(\Delta_{r-1})$ we can find a ball
  containing $x$ and 
  $p_\alpha(x)<\epsilon_1+c\epsilon_2$ for some constant $c$.
  This implies $p_\alpha(x)=0$, $\forall x\in{\rm int}(\Delta_{r-1})$,
  but it contradicts to the fact that
  $\langle\nu_{\alpha},x^a\rangle>0$, $\forall a\in I$
  followed by the expression \eqref{theo:stat_k:eq1}, because
  the killing rate $\tilde{k}(\tilde{a})$ is bounded
  and the Markov time satisfies
  $\tilde{\mathbf{P}}(\tau<\infty)=1$ by
  Proposition~\ref{prop:stop}.
\end{proof}  
    
An immediately consequence is the following corollary,
which is an analogous result to Corollary~\ref{coro:mom}.

\begin{corollary}\label{coro:mom_k}
  The moments of the stationary state $\nu_\alpha$ of
  the adjoint Markov semigroup $\{\tilde{T}^*_t\}$ are positive, namely,
  \begin{equation}\label{moment_k}
    m_{a}(\alpha)
    :=\langle\nu_\alpha,x^a\rangle=\tilde{\mathbf{E}}_{\nu_\alpha} x^a>0,
    \quad \text{for~each} ~~ a \in I.
  \end{equation}
\end{corollary}

The moments of the stationary state can be obtained by the condition
for the stationary state \eqref{stat}. It gives a system of
recurrence relations:
\begin{equation}
  0=\sum_{i\in V}\sum_{j\in N(i)}\frac{a_i(a_i-1)}{2}
  m_{a-e_i+e_j}
  -\sum_{i\in V}\sum_{j\in N(i)}\frac{a_ia_j}{2}m_a
  +\frac{\alpha}{2}\sum_{i\in V}a_i(m_{a-e_i}-rm_a)
  \label{mom_k}
\end{equation}  
for each $a\in I$ with the boundary condition $m_0=1$.
In contrast to the system of ordinary differential
equations \eqref{mom}, this system is not closed
among the moments of the same order. In prior to solve
the system for the moment of a given monomial $a$, we
have to solve the systems for the moments of lower orders
than $a$. Therefore, it seems a formidable task to
solve the system \eqref{mom_k}. The diffusion on
a complete graph is an exception.

\begin{example}\label{exam:dir-multi}
  Let $\mathcal{G}=K_r$, which is the complete graph
  consisting of $r$ vertices discussed in
  Example~\ref{exam:K}. The unique solution of
  the system of recurrence relations \eqref{mom_k} is
  \begin{equation}\label{exam:dir-multi:eq1}
    m_a(\alpha)=\frac{\prod_{i=1}^r(\alpha)_{a_i}}
    {(r\alpha)_{n}}, \quad
    \forall a\in I.
  \end{equation}
  Moreover, since this expression is the moments of
  the symmetric Dirichlet distribution of parameter
  $\alpha$, the stationary state is the Dirichlet
  distribution:
  \[
    \nu_\alpha(x_1,\ldots,x_r)=
    \frac{\Gamma(r\alpha)}{\{\Gamma(\alpha)\}^r}
    \prod_{i=1}^r x_i^{\alpha-1}dx_1\cdots dx_r.
  \]  
\end{example}

\begin{remark}\label{rema:esf}
  The limit of the moments \eqref{exam:dir-multi:eq1}
  is known as the Dirichlet-multinomial distribution
  up to multiplication of the multinomial coefficient.
  Renumbering the set ${\rm supp}(a)$ by
  $\{1,\ldots,l\}$ and taking the limit $\alpha\to 0$
  and $r\to\infty$ with $r\alpha=\theta>0$,
  the expression \eqref{exam:dir-multi:eq1} reduces to
  the form
  \begin{equation}\label{esf}
    \frac{\theta^{l}}{(\theta)_n}\prod_{i=1}^r(a_i-1)!,
    \quad \forall a\in\Pi_{n,l},
  \end{equation}
  which is known as the Ewens sampling formula, or
  the exchangeable partition probability function of
  the Dirichlet prior process in Bayesian statistics
  (see, e.g., \cite{Mano18} for an introduction).
  Karlin and McGregor \cite{KM72} derived this formula
  by using a system of recurrence relations based on
  coalescents mentioned in Remark~\ref{rema:coa}. In
  this sense, we have found an alternative system of
  recurrence relations \eqref{mom_k} the formula
  \eqref{esf} satisfies based on collisions.
\end{remark}

\section{Applications}
\label{sect:appl}

In this section, we present applications of developed
results in previous sections. 

\subsection{Finding independent sets of graphs}

Itoh et al.'s discrete stochastic model described in
Section~\ref{sect:intr} stops when the set of vertices
occupied by at least one particle constitutes
an independent set of a graph. The model is summarized
as the following procedure.

\begin{algorithm}~\label{algo:fin}
  \begin{itemize}
  \item[] Input: A graph $\mathcal{G}=(V,E)$ with
    vertices $V=\{1,\ldots,r\}$, the number of particles
    $N\ge r$, and an integer $M\in\mathbf{N}$ to stop
    the iteration.
  \item[] Output: A candidate of an independent set of
    $\mathcal{G}$.
  \end{itemize}
  \begin{itemize}
  \item[] Step 1: Assign particles to vertices such that
    at least one particle is assigned to each vertex.  
  \item[] Step 2: Set $c=0$.
  \item[] Step 3: Choose two distinct particles uniformly
    random. Let $i\in V$ and $j\in V$ be the vertices to
    which the particles are assigned.
  \item[] Step 4: If $i\sim j$, then assign both particles
    to $i$ or $j$ with probability $1/2$, and go to Step
    2. Otherwise, $c\leftarrow c+1$.
  \item[] Step 5: If $c<M$, go to Step 3.
  \item[] Step 6: Output the list of vertices to which at
    least one particle is assigned. 
  \end{itemize}
\end{algorithm}

The cardinality of the set of vertices to which at least
one particle is assigned decreases, and the set eventually
reduces to an independent set of $\mathcal{G}$.
The integer $M$ is needed to confirm that we cannot choose
particles from neighboring vertices. If $M$ is sufficiently
large, Algorithm~\ref{algo:fin} provides an independent
set with high probability.

A natural question is how many steps is needed to find
an independent set. To answer this question seems
hard, but regarding the diffusion satisfying
the stochastic differential equation \eqref{sde} as
an approximation of the procedure of Algorithm~\ref{algo:fin},
we can deduce some rough idea. Because of the scaling in
the diffusion limit, the unit time in the diffusion
corresponds to the $N(N-1)/2$ iterations of Steps 3 and
4 of Algorithm~\ref{algo:samp}.

According to the argument of Proposition~\ref{prop:start},
a sample path of the diffusion $(x(t),P_x)$ starts form
a point $x\in{\rm int}(\Delta_{r-1})$ is absorbed into
lower dimensional faces and eventually absorbed into
a face corresponding to an independent set.

\begin{proposition}\label{prop:alg}
  Let $U\subset V$ be a set of vertices which is not
  an independent set of a graph $\mathcal{G}$. For
  a sample path of the diffusion $(x(t),P_x)$ starts from
  a point $x\in{\rm int}({\rm conv}\{e_i:i\in U\})$,
  the Markov time
  \begin{equation}
    \tau_U=
    \inf\{t>0:x(t)\in{\rm Bd}({\rm conv}\{e_i:i\in U\})\}
    \label{exit}
  \end{equation}
  satisfies
  \[
    \mathbf{P}_x(\tau_U>t) \ge c_x e^{-t |E_U|},
    \quad t>0,
  \]  
  where $E_U$ is the edge set of the induced subgraph
  of $\mathcal{G}$ consisting of $U$,
  ${\rm Bd}({\rm conv}\{e_i:i\in U\})$ is the boundary
  of ${\rm conv}\{e_i:i\in U\}$, and $c_x\in(0,1]$ is
  a constant depending on $x$.
\end{proposition}

\begin{proof}
  By the argument in the proof of Theorem~\ref{theo:stat},
  we have
  \begin{align*}
    \mathbf{E}_x(x(t))&=\mathbf{E}_x\{x(t);x(t)>0\}
    \le\max(x)\mathbf{E}_x\{1(x(t)>0)\}\\
    &=\max(x)
    \mathbf{P}_x
    \{x(t)\in{\rm int}({\rm conv}\{e_i:i\in U\})\}\\
    &=\max(x)\mathbf{P}_x(\tau_U>t), \quad
    x=\prod_{i\in U}x_i,
  \end{align*}
  while $\mathbf{E}_x(x(t))=x e^{-t |E_U|}$.
  Choose $c_x=x\{\max(x)\}^{-1}=x|U|^{|U|}$. 
\end{proof}

The author does not find any other property of
the Markov time $\tau_U$ for generic graphs, but
the diffusion on a complete graph is an exception;
the probability distribution function can be
obtained exactly.

\begin{proposition}\label{prop:exit}
  Let $\mathcal{G}$ be a complete graph $K_r$.
  For a face $\Delta_{s-1}={\rm conv}\{e_i:i\in U\}$,
  $2\le s\le r$, the distribution of the Markov time
  \eqref{exit} is represented as
  \begin{align*}
    \mathbf{P}_x(\tau_U>t)
    &=\sum_{i\ge s}(2i-1)(-1)^ie^{-\frac{i(i-1)}{2}t}\\
    \times&\sum_{l=s-2}^{i-2}\frac{(2-i)_l(i+1)_l}{l!(l+1)!}
    \left\{
      \sum_{a\in\Pi_{l+2,s}}
      \left(\begin{array}{c}l+2\\a\end{array}\right)x^a
      -\sum_{b\in\Pi_{l+1,s}}
      \left(\begin{array}{c}l+1\\b\end{array}\right)x^b
    \right\}.
  \end{align*}
\end{proposition}

\begin{proof}
  The inclusion-exclusion argument shows the following
  expression:
  \begin{align}
    \mathbf{P}_x(\tau_U>t)=&
    \bar{p}_{1,\ldots,s}(t)
    -\sum_{i_1,\ldots,i_{s-1}}\bar{p}_{i_1,\ldots,i_{s-1}}(t)
    +\sum_{i_1,\ldots,i_{s-2}}\bar{p}_{i_1,\ldots,i_{s-2}}(t)
    \nonumber\\
    &-\cdots+(-1)^{s-1}\sum_{i=1}^s\bar{p}_{i}(t),
   \label{prop:exit:eq1} 
  \end{align}
  where $\bar{p}_{i_1,\ldots,i_u}(t)$ is the total
  mass on ${\rm conv}\{e_{i_1},\ldots,e_{i_u}\}$ and
  the summations are taken over the totality of
  distinct indices chosen from $\{1,\ldots,s\}$. For
  $K_2$, an explicit expression can be obtained by
  solving a backward equation for the diffusion
  (equation (4.15) in \cite{Kimura64}):
  \begin{equation}
    \bar{p}_1(t)=x_1+\sum_{i\ge 2}
    (2i-1)(-1)^ie^{-i(i-1)t/2}
    \sum_{l\ge 0}
    \frac{(2-i)_l(i+1)_l}{l!(l+1)!}(x_1^{l+2}-x_1^{l+1}).
    \label{prop:exit:eq2}
  \end{equation}
  A complete graph can be reduced to $K_2$ by any
  partition of the vertex set to two vertex sets
  (the reduction is defined in Section~\ref{sect:intr}).
  For example, $K_3$ is reducible to $K_2$ consisting
  of vertices $\{1+2,3\}$, where the vertex $1+2$ is
  obtained by identifying vertices 1 and 2. In the same way,
  $K_3$ is also reducible to $\{1+3,2\}$ and $\{2+3,1\}$.
  Therefore, an expression of $\bar{p}_{i_1,\ldots,i_u}(t)$ is
  obtained by the right hand side of \eqref{prop:exit:eq2}
  with replacing $x_1$ by $x_{i_1}+\cdots+x_{i_u}$.
  The inclusion-exclusion argument gives
  \begin{align}
  \sum_{a\in \Pi_{l,s}}
  \left(\begin{array}{c}l\\a\end{array}\right)x^a=    
  &(x_1+\cdots+x_s)^l
  -\sum_{i_1,\ldots,i_{s-1}}(x_{i_1}+\cdots+x_{i_{s-1}})^l+\cdots
  \nonumber\\
  &+(-1)^{s-1}\sum_{i=1}^s(x_1+\cdots+x_s),
  \label{prop:exit:eq3}
  \end{align}
  where the both sides are zero if $l<s$.
  Substituting the expression obtained by the expression
  \eqref{prop:exit:eq2} into \eqref{prop:exit:eq1}
  and collecting terms by using the equality
  \eqref{prop:exit:eq3}, the assertion follows.
\end{proof}

According to Proposition~\ref{prop:exit},
the probability distribution function of exit time from
${\rm int}(\Delta_{s-1})$ is asymptotically
$1-(2s-1)!!2^{s-1}xe^{-s(s-1)t/2}$ for large $t$,
where $(2s-1)!!=(2s-1)(2s-3)\cdots1$.
Let a sequence of vertex sets occupied by at
least one particle in Algorithm~\ref{algo:samp} be
denoted by $V$, $U^{(r-1)}$, $U^{(r-2)}$, $\ldots$,
$U^{(1)}=\{i\}$ for a vertex $i\in V$. If the exit time
for $U^{(s)}$ followed the exponential distribution
of mean $|E_{U^{(s)}}|=s(s-1)/2$ (of course, not exactly
true), the expectation of the waiting time of
a sample path is absorbed into the vertex would be
$2(1-r^{-1})$ for large $r$. This rough argument
suggests that the expectation of the computation cost
of Algorithm~\ref{algo:samp} would be $O(rN^2)$ for
a complete graph, because an iteration of Steps 3 and
4 can be executed in $O(r)$. Luby's algorithm for
finding an independent set described in
Section~\ref{sect:intr} demands $O(1)$ with using
$O(r^2)$ processors.

\subsection{Bayesian graph selection}

Consider sampling of particles of size $n$ from
the unique stationary state of the adjoint Markov semigroup
$\{\tilde{T}^*\}$ on $\mathcal{P}(\Delta_{r-1})$
induced by the Markov semigroup $\{T_t\}$ associated
with the diffusion $(\tilde{x}(t),\tilde{P}_x)$ appeared in
Theorem~\ref{theo:stat_k} such that $a_i$ particles
of a graph $\mathcal{G}=(V,E)$ are taken from
the vertex $i\in V$. We assume the probability of
taking a sample does not depend on the order of
particles, namely, exchangeable. Such probabilities
constitute the multinomial distribution, namely,
a probability measure on ordered non-negative integer
partitions of $n$ as
\begin{equation*}\label{sample}
  q_a:=\left(\begin{array}{c}n\\a\end{array}\right)
  x^a, \quad a\in \Pi_{n,r}^{\ge 0},
  \quad r=|V|,
\end{equation*}
satisfying $\sum_a q_a(t)=1$. The moment $m_a$ defined
by \eqref{moment_k} is the expectation of
the {\it sample probability} $q_a$ up to
the multinomial coefficient:
\begin{equation}
  \tilde{\mathbf{E}}_{\nu_\alpha} q_a
  =\left(\begin{array}{c}n\\a\end{array}\right)
  \tilde{\mathbf{E}}_{\nu_\alpha} x^a
  =\left(\begin{array}{c}n\\a\end{array}\right)
  m_a(\alpha).
  \label{mlike}  
\end{equation}

Before proceeding to discuss computational
issue, we present a motivating problem in Bayesian
statistics. The expected sample
probability \eqref{mlike} is a mixture of multinomial
distributions of parameters $x$ over the stationary
state $\nu_\alpha$ of the adjoint Markov semigroup
$\{\tilde{T}^*\}$. In statistical terminology,
the sample probability $q_a$ and the expectation
\eqref{mlike} are called the likelihood and
the {\it marginal likelihood,} respectively, and
$\nu_\alpha$ is called the {\it prior distribution}
for the parameter $x\in\Delta_{r-1}$.

Suppose we are interested in selecting a graphical
model consisting of four vertices from three
candidate models: a star graph $S_3$, a cycle graph
$C_4$, and a complete graph $K_4$ (Figure 3).
\begin{figure}
  \begin{center}
  \includegraphics[width=0.7\textwidth]{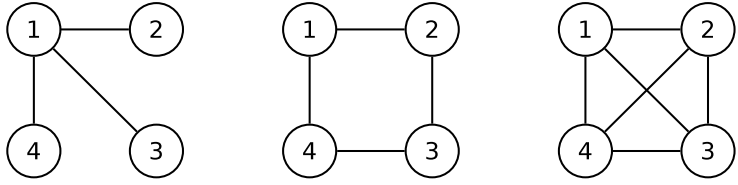}
  \caption{Three candidate graphical models:
    a star graph $S_3$, a cycle graph $C_4$, and
    a complete graph $K_4$.} 
  \label{fig3}
  \end{center}
\end{figure}
For this purpose, we employ stationary states
$\nu_\alpha$ as the prior distributions. As we have
seen in Example~\ref{exam:dir-multi}, the prior
distribution for $K_4$ is the Dirichlet distribution,
but for the prior distributions for other graphs closed
form expressions of the distribution function are 
not known. Suppose we have a sample consisting of
two particles. If it is $e_1+e_3$, by solving
the recurrence relation \eqref{mom_k}, we obtain
the expected sample probabilities under $S_3$, $C_4$,
and $K_4$ as
\[
  \frac{\alpha}{2(1+4\alpha)}, \quad \frac{1}{8},
  \quad \text{and} \quad \frac{\alpha}{2(1+4\alpha)}, 
\]
respectively. On the other hand, if the sample is
$e_2+e_4$, they are
\[
  \frac{1}{8}, \quad \frac{1}{8}, \quad \text{and}
  \quad \frac{\alpha}{2(1+4\alpha)}.
\]
If $\alpha$ is small, $e_1+e_3$ supports $C_4$, while
$e_2+e_4$ does not support $K_4$. This is reasonable,
because the vertices set $\{1,3\}$ is an independent
set of $C_4$, but not an independent set of $S_3$ and
$K_4$. On the other hand, the set $\{2,4\}$ is not
an independent set for $K_4$, but an independent set
of $S_3$ and $C_4$. The ratio of marginal likelihoods
is called the Bayes factor, which is a standard model
selection criterion in Bayesian statistics (see, e.g.,
Section 6.1 of \cite{BS94}). If a sample is $e_1+e_3$,
the Bayes factor of $C_4$ to $S_3$ or $K_4$ are
$1+1/(4\alpha)$. Therefore, $C_4$ is supported if $\alpha$
is small, while all graphs are equivalent if $\alpha$
is large. We do not discuss details of statistical
aspects including how to choose $\alpha$, but it is
worthwhile to mention that positive $\alpha$ improves
stability of model selection especially for small
sample. In fact, if $\alpha$ is small, the Bayes factor
drastically changes by adding a sample unit.
Suppose we have a sample $e_1+e_3$ and
take an additional sample unit. If it is $e_2$,
the expected sample probabilities for the sample
$e_1+e_2+e_3$ under $S_3$, $C_4$, and $K_4$
are
\[
  \frac{3\alpha(1+12\alpha)}{32(1+3\alpha)(1+4\alpha)}, \quad
  \frac{3\alpha(1+12\alpha)}{32(1+3\alpha)(1+4\alpha)},
  \quad \text{and} \quad \frac{3\alpha^2}{4(1+2\alpha)(1+4\alpha)}, 
\]
respectively. The Bayes factor of $C_4$ to $S_3$ is
unity, which means that the graphs $C_4$ and $S_3$
are equivalent. This conclusion is quite different
from that deduced from the sample $e_1+e_3$.
In the limit of large $\alpha$, by
Corollary~\ref{coro:stat_k}, the expected sample
probability of any graph follows the unique limit
distribution, the multinomial distribution.

Now let us discuss computation of expected sample
probabilities. A closed form expression of
the stationary state $\nu_\alpha$ of the adjoint
Markov semigroup $\{\tilde{T}^*\}$ is not known
for generic graphs, nevertheless, we can compute
the expected sample probabilities of any graph by
solving the system of recurrence relations
\eqref{mom_k}. Solving the system becomes prohibitive
as the sample size $n$ grows, but following algorithm,
which is a byproduct of Theorem~\ref{theo:stat_k},
provides an unbiased estimator of the expected sample
probability.

\begin{algorithm}~\label{algo:samp}
  \begin{itemize}
  \item[] Input: A sample taken from a graph
    $\mathcal{G}=(V,E)$ consisting of a vector
    $a\in\Pi_{n,r}$, $r=|V|$ and the parameter value
    $\alpha>0$.
  \item[] Output: A random variable following
    $\exp\{-\int_0^{\tau}\tilde{k}(\tilde{a}(s))ds\}$
    appeared in \eqref{theo:stat_k:eq1}.
  \end{itemize}
  \begin{itemize}
  \item[] Step 1: Set $k=0$.
  \item[] Step 2: Get a random number $T$ following
    the exponential distribution of mean
    $\sum_{i=1}^r d_ia_i(a_i-1)/2+\alpha|a|/2$.
  \item[] Step 3: $k\leftarrow k+\tilde{k}(a)T$.    
  \item[] Step 4: Divide the interval $[0,1]$ with
    the ratio\\
    $a_1(a_1-1): \cdots :a_1(a_1-1):
     a_2(a_2-1): \cdots :a_2(a_2-1):\cdots:
     a_r(a_r-1): \cdots :a_r(a_r-1):
     \alpha a_1: \cdots :\alpha a_r$\\
    such that $a_i(a_i-1)$, $i\in \{1,\ldots,r\}$
    appears $d_i$ times.
  \item[] Step 5: Get a random number $U$ following
    the uniform distribution on $[0,1]$.
  \item[] Step 6: If $U$ falls in the $j$-th interval
    of $a_i(a_i-1)$, let $a_i\leftarrow a_i-1$,
    $a_j\leftarrow a_j+1$. Otherwise, if $U$ falls in
    the interval of $\alpha a_i$, let $a_i\leftarrow a_i-1$.
  \item[] Step 7: If $|a|=0$, output $e^{-k}$.
    Otherwise, go to Step 2.
  \end{itemize}
\end{algorithm}

By Corollary~\ref{coro:mom_k}, the output of
Algorithm~\ref{algo:samp} is an unbiased estimator of
$m_a(\alpha)$, which gives the expected sample
probability \eqref{mlike} by multiplying
the multinomial coefficient.

An attractive property of Algorithm~\ref{algo:samp}
as a Markov chain Monte Carlo is that it is a direct
sampler, namely, it generates random variables
independently and exactly follow the target
distribution. In fact, this algorithm can be
regarded as a variant of a direct sampling algorithm
called coupling from the past (see, e.g., Chapter 25
of \cite{LP17} for a concise summary). Regarding
a sample as being generated from the infinite past
and the time is going backward, the time when all
particles are erased is the time when the sample path
can be regarded as that comes from infinite past,
because the sample path does not depend on any events
older than the time. We have the following estimate
of steps needed to complete the procedure.

\begin{proposition}\label{prop:samp}
  For a sample of size $n$, the expected number of
  steps needed to complete Algorithm~\ref{algo:samp}
  to obtain an unbiased estimator of the expected
  sample probability \eqref{mlike} is
  $O\{|E|(n+n(n-1)/(2\alpha))\}$ for large $|E|$.
\end{proposition}  

\begin{proof}
  For a state $a$ satisfying $|a|=m$, the probability
  that a particle is erased at the next step is bounded
  below:
  \[
  \frac{\alpha m}{\sum_{i=1}^r a_i^2+(\alpha-1)m}
  \ge \frac{\alpha}{m+\alpha-1}.
  \]
  Therefore, the waiting time of an erase of a particle
  is stochastically smaller than the waiting time of
  an event following the geometric distribution of
  waiting time $(m+\alpha-1)/\alpha$. Sum of the
  waiting times from $m=1$ to $m=n$ is
  $O(n+n(n-1)/(2\alpha))$. Steps 4 and 6 demand
  $d_1+\cdots+d_r+r=2|E|+r=O(|E|)$ steps for large
  $|E|$. Therefore, the assertion follows.
\end{proof}  

We have focused on the diffusion with drift, but
moments of the marginal distribution of the diffusion
without drift at a given time, i.e., \eqref{moment},
can be computed by an analog to
Algorithm~\ref{algo:samp}. The problem reduces to
solving the system of differential equations
\eqref{mom}. We omit the discussion, but a similar
problem for a complete graph $K_4$ was discussed in
\cite{Mano13}.

\section{Discussion}\label{sect:disc}

We have studied diffusions taking value in probability
measures on a graph whose masses on each vertices satisfy
the stochastic differential equations of the forms
\eqref{sde} and \eqref{sde_k} by using their dual Markov
chains on integer partitions and on finite integer
lattices, respectively. Many problems remain to
be solved, especially for \eqref{sde}. First of all,
a formal proof of the existence of the semigroup
$\{T_t\}$ associated with the generator \eqref{gen}
should be established, which demands pathwise uniqueness
of the solution of \eqref{sde}. As we have emphasized in
the text, some arguments, especially those after
Propositions~\ref{prop:abs} and \ref{prop:exit}, are
rough and restrictive. They could be improved.
Stationary states of the Markov semigroup need further
studies. A counter part of Theorem 1.5 of \cite{SU86} or
Theorem~\ref{theo:stat2} on regularity of stationary
states could be established by detailed analysis of
the diffusion. Further properties of the diffusion,
such as absorption probabilities into a stationary state
and the waiting times are interesting. To obtain explicit
expressions of them are challenging, but such expressions
would be helpful for further understanding of the diffusion.

Two applications of the diffusions are discussed:
analysis of an algorithm to find an independent set
of a graph, and a Bayesian graph selection based on
computation of expected sample probability by using
coupling from the past. Further applications and
targets for modeling may exist.

For coalescents mentioned in Remark~\ref{rema:coa},
properties of a ``genealogy'' of a sample, which is a
sample path of the dual Markov chain, are intensively
studied because a genealogy itself is used as a stochastic
model of DNA sequence variation (see, e.g., \cite{Durrett08}).
Random graphs such as Figures 1 and 2 are counterparts
of such genealogies. Study of properties of such random
graphs would be interesting.

\section*{Acknowledgements}
The author thanks Prof. Yoshiaki Itoh for introducing
there work \cite{IMS98} to him.
An earlier version of this work was
presented at a workshop on coalescent theory at
Centre de Recherches Math\'ematiques, University of
Montreal, in October 2013.
%He thanks organizers for their hospitarity.
%He also thanks Editors of this issue for
%giving me the opportunity to present this work here.
The author is supported in part by JSPS KAKENHI Grant
20K03742 and was supported in part by JST Presto Grant,
Japan.

\begin{flushleft}

Shuhei Mano\\
The Institute of Statistical Mathematics, Tokyo 190-8562, Japan\\
E-mail: smano@ism.ac.jp

\end{flushleft}

\end{document}